\documentclass[11pt, epsfig]{article}
\usepackage{epsfig, amsmath, amssymb, amsthm, times,appendix}
\usepackage{graphicx}

\newcommand{\R}{{\mathbb R}}

\newcommand{\N}{{\mathbb N}}
\newcommand{\Z}{{\mathbb Z}}

\def\1{{\mathbf 1}}
\def\Gn{\Gamma_n}

\DeclareMathOperator{\esssup}{ess\,sup}

\usepackage[mathcal]{eucal}

\newtheorem {thm}{Theorem}[section]
\newtheorem {lem}[thm]{Lemma}

\newtheorem {cor}[thm]{Corollary}

\theoremstyle{defintion}
\newtheorem {df}[thm]{Definition}

\theoremstyle{rem}
\newtheorem{rem}[thm]{Remark}

\theoremstyle{example}
\newtheorem{ex}[thm]{Example}

\theoremstyle{assumption}

\def\E{\operatorname{\mathbb E}}
\def\P{\operatorname{\mathbb P}}
\def\T{{\mathbb T}}

\def\lbl{\label}
\def\be{\begin{equation}}
\def\ee{\end{equation}}
\def\p{\partial}

\def\Gn{\Gamma_n}
\def\un{\mathbf{u}_n}
\def\bu{\mathbf{u}}
\def\bv{\mathbf{v}}

\title{The continuum limit of the Kuramoto model on sparse random graphs}
\author{
Georgi S. Medvedev\thanks{Department of Mathematics, 
Drexel University, 3141 Chestnut Street, Philadelphia, PA 19104,
{\tt medvedev@drexel.edu}
} 
}
\begin{document}
\maketitle

\begin{abstract}
In this paper, we study convergence of coupled dynamical systems on convergent 
sequences of graphs to a continuum limit. We show that the solutions of the initial value
problem for the dynamical system on a convergent graph sequence tend to that for the nonlocal diffusion
equation on a unit interval, as the graph size tends to infinity. We improve our earlier results in
[Medvedev, The nonlinear heat equation on W-random graphs, Arch. Rational Mech. Anal., 212(3), pp. 781–803]
and extend them to a larger class of graphs, which includes directed and undirected,
sparse and dense, random and deterministic graphs.

There are three main ingredients of our approach. First, we employ a flexible framework 
for incorporating random graphs into the models of interacting dynamical systems, which fits 
seamlessly with the derivation of the continuum limit. Next, we prove the averaging principle
for approximating a dynamical system on a random graph by its deterministic (averaged) counterpart.
The proof covers systems on sparse graphs and yields almost sure convergence on time intervals
of order $\log n,$ where $n$ is the number of vertices. Finally, a Galerkin scheme is developed to show
convergence of the averaged model to the continuum limit. 

The analysis of this paper covers the Kuramoto model of coupled phase oscillators on a
variety of graphs including sparse Erd\H{o}s-R{\' e}nyi, small-world, and power law graphs.
\end{abstract}

\noindent\textbf{Keywords:}~  continuum limit, random graph, sparse graph,  graph limit, Galerkin method.


\section{Introduction}
\setcounter{equation}{0}

Understanding principles of collective dynamics in large ensembles of
interacting dynamical systems is a fundamental problem in nonlinear science with applications
ranging from neuronal and genetic networks to power grids and the Internet. The key distinction
of coupled dynamical systems considered in this paper from classical spatially extended systems 
such as partial differential equations or lattice dynamical systems is that the spatial domain of 
the former class of models is a general graph. Given an enormous variety of graphs and their
complexity, analyzing dynamical systems on large and, in particular, on random graphs is a 
challenging problem.

In \cite{Med14a, Med14b}, we initiated a study of the continuum limit of systems of coupled 
phase oscillators on convergent families of graphs. We used the fact that a large class of 
(dense) graphs, including many of those of interest in applications, can be conveniently described
analytically by a measurable function on a unit square, called a graphon \cite{LovGraphLim12}.
Roughly speaking, a graphon represents a limit of the adjacency matrix of a graph as its
size tends to infinity. Using graphons, we were able to derive and justify  the continuum limit
for the Kuramoto model (KM) on a great variety of graphs, which led to new studies of the KM on 
nontrivial graphs \cite{Med14c, MedTan15b, MedWri17, CMM18}. Importantly, the same approach can be 
successfully applied to justify the mean field limit for coupled dynamical systems 
on graphs \cite{ChiMed17a, KVMed18}.

The analysis in \cite{Med14a, Med14b} did not cover the KM on sparse graphs.
The progress in this direction became possible with the theory of $L^p$ graphons 
used to define graph limits for sparse graphs of unbounded degree \cite{BCCZ}. Using the 
insights from \cite{BCCZ}, we addressed the problem of the continuum limit of the KM
 on sparse graphs in  \cite{KVMed17}. 
While we were able to extend  many of our techniques to the KM on  a large class
of sparse graphs (including power-law graphs), some of the results in \cite{KVMed17} apply only 
to systems with linear diffusion. In the present work, we unify and, 
in the case of the KM on random graphs, significantly improve the results in 
\cite{Med14a, Med14b, KVMed17}. The contribution of this paper is twofold. First, we propose
a flexible framework for describing directed and undirected, sparse and dense, random and
deterministic graphs to be used in interacting dynamical systems models. 
This framework naturally leads to continuum models approximating dynamical systems on large 
graphs. Second, we refine our techniques 
to obtain stronger results on convergence to the continuum limit, which, in addition, apply to 
a wider class of graphs than in \cite{Med14b, KVMed17}. Even for the KM on dense graphs, our
results are much stronger:  we show convergence of solutions on the time intervals
of order $\log n$, compared to finite intervals in \cite{Med14b} 
\footnote{The very last step in the proof of Theorem~3.3 of \cite{Med14b} estimating 
$\P\left( \sup_{t\in [0,T]} \|z_n(t)\|^2_{2,n} > Cn^{-1} \right)$ is incorrect (see \cite{Med18}
for corrections).
}.
Furthermore, in the present work,
convergence is shown with probability $1$ versus convergence in probability in the earlier papers  
\cite{Med14b, KVMed17}. 
Finally, our results 
apply to the KM on sparse directed graphs, which have not been considered in 
\cite{Med14b, KVMed17}. Taken together, the results of this paper reveal a fuller potential of 
our method for proving convergence of discrete problems on graphs to a continuum limit.

As in \cite{Med14b}, the main result of this work is the proof of convergence of solutions
of the initial value problems (IVPs) for the KM on graphs to the solution of that  for the limiting 
nonlocal diffusion equation as the size of the graph tends to infinity. 
In its most basic version, the result may be seen as convergence of numerical discretization
of a nonlocal diffusion equation. The contribution of this paper, however, is much deeper and 
more interesting. For starters, we consider dynamical problems on random graphs. This 
situation is not treated in classical numerical analysis. More importantly, we use minimal regularity
assumptions on the limiting graphon $W$. The only assumption is that $W$ is a  square integrable 
function on a unit square. This allows us to treat a huge class of graphs 
and affords great flexibility in applications.  The fact that $W$ does not require any regularity 
beyond integrability means, in particular, that the order, in which vertices are sampled, is irrelevant.
Last but not least, the convergence problem analyzed in this work is motivated by concrete 
questions about the dynamics of large networks 
\cite{WilStr06, MedTan17, CMM18}.

There are three main ingredients in our proof of convergence. First, as we commented above, 
we construct convergent families of graphs in the spirit of W-random graphs \cite{LovSze06}. 
This description covers a broad class of graphs and  fits seamlessly with the 
analysis of convergence of the discrete models to the continuum limit.
In particular, the limit of  the graph sequence, given by a measurable real-valued
function $W$ on the unit square, is used later in the derivation of the continuum 
model as a kernel of a nonlocal diffusion term.  Many random 
graph models like small-world,
Erd\H{o}s-R{\' e}nyi, and even power law graphs have relatively simple graph limits,
 which makes the corresponding continuum models amenable to analysis 
\cite{Med14b, MedWri17, MedTan17}. 
The key tool for dealing with the models on random graphs is the
averaging principle, which justifies approximation of  a coupled system on a random graph 
by an averaged deterministic model on a complete weighted graph.
Finally, the proof of convergence of the discrete deterministic models to the continuous 
one employs 
the interpretation of the discrete problems as Galerkin approximation of the continuum
limit (cf.~\cite{Eva-PDE}). The Galerkin method is used to show existence and uniqueness
of the weak solution of the IVP for the continuous problem. The fixed point argument used 
in \cite{Med14a, KVMed17} does not apply the more general problem considered in this paper.

The organization of the paper is as follows. In the next section, we define convergent
graph sequences that are used in the remainder of this paper and formulate the 
KM on random graphs.
In Section~\ref{sec.main}, we state the main result about the convergence of the discrete model
on graphs to the continuum model. Here, we also explain the main steps of the proof.
 In Section~\ref{sec.average}, we prove the averaging principle, the first main ingredient of the 
proof of convergence to the continuum limit. It allows to approximate 
the KM  on a random graph by a deterministic model via averaging over all realizations of the 
random graph model. The averaged model then suggests the continuum limit in the 
form of a nonlinear nonlocal diffusion equation.
In Section~\ref{sec.continuum}, we introduce  Galerkin approximation of the continuum model
and state  Theorem~\ref{thm.main} about the convergence of the Galerkin scheme.
The use of the Galerkin method is twofold. First, it establishes the wellposedness of the 
IVP for the continuum model. Second, it is used to show convergence of the discrete models
to the continuum limit. This is the second ingredient of our method. Together with the averaging principle
and some auxiliary estimates, it implies the convergence of the KM on random graphs.
Section~\ref{sec.converge} presents the proof on the convergence of the Galerkin method.

\section{The KM on graphs}\lbl{sec.KM}
\setcounter{equation}{0}

Let $\Gn=\langle V(\Gn), E_d(\Gn), A_n\rangle$ be a weighted directed graph on $n$ nodes.
$V(\Gn)=[n]$ stands for the node set of $\Gn$.  $A_n=(a_{n,ij})$ is an
$n\times n$ weight matrix. The edge set
$$
E_d(\Gn)=\left\{ (i,j)\in [n]^2:\;a_{n,ij}\ne 0\right\}.
$$
An edge $(i,j)$ is an ordered pair of nodes. We will also use $j\to i$
to denote the edge $(i,j)$. Loops are allowed.

We will also consider undirected weighted graphs $\Gn=\langle V(\Gn),
E(\Gn), A_n\rangle$. In this case, $A_n$ is a symmetric matrix and the edges are 
unordered pairs of nodes
$$
E(\Gn)=\left\{ \{i,j\}\in [n]^2:\;a_{n,ij}\ne 0\right\}.
$$
We will use $i\sim j$ as a shorthand for $\{i,j\}\in E(\Gn)$.

 Consider a system of coupled oscillators on a sequence of weighted (directed or undirected) graphs 
$\Gn$
\begin{eqnarray}
\lbl{KM}
\dot u_{n,i} &=& f(u_{n,i},t) + (n\alpha_n)^{-1} \sum_{j=1}^n a_{n,ij} D(u_{n,j}-u_{n,i}),
\quad i\in [n],\\
\lbl{KM-ic}
u_{n,i}(0)&=&u_{n,i}^0. 
\end{eqnarray}
Here, $u_{n,i}:\R\to\T:=\R/2\pi\Z$ stands for the phase of oscillator $i\in [n]$ as a function of time.
$D$ is a $2\pi$-periodic Lipschitz continuous function, $\operatorname{Lip}(D)=L_D$. 
Without loss of generality,
\be\lbl{boundD}
\max_{u\in \T} |D(u)|=1.
\ee
Function 
$f(u,t)$ is a Lipschitz continuous in $u$, $\operatorname{Lip}_u(f)=L_u$,
and continuous in $t$. 
The sum on the right-hand side of \eqref{KM} models the interaction between oscillators.
Finally, unless otherwise specified $\alpha_n=1$. The scaling factor $\alpha_n$ will be
 needed for the KM on sparse random graphs, as explained below.

Equation \eqref{KM} generalizes the  original KM by allowing nonlinearity $f(u,t)$ and 
sequence $\{\Gamma_n\}$ as a spatial domain.
We are interested in the large $n$ limit of \eqref{KM}, \eqref{KM-ic}. One can expect a limiting 
behavior of solutions of \eqref{KM}, \eqref{KM-ic}, only if the graph sequence $\{\Gn\}$ 
has a well defined asymptotic behavior in the limit
as $n\to\infty$. We define  the asymptotic structure of $\{\Gn\}$  using function 
$W\in L^1(I^2)$, called a graphon. To define $\{\Gn\}$, we  discretize the unit interval 
by points $x_{n,j}=j/n, \; j\in \{0\}\cup [n]$ and denote
$
I_{n,i}:=(x_{n,i-1}, x_{n,i}], \; i\in [n].
$
We chose the uniform mesh $\{x_{n,i}, \; i=0,1,\dots,n\},\; n\in \N,$ for simplicity, as this is 
sufficient for the applications we have in mind. In general, any 
dense set of points from $[0,1]$ can be used. In particular, one could use random points sampled 
from the uniform distribution on $[0,1],$ as was done in \cite{Med14b}.

The following constructions are used to model a variety of dense and sparse, directed and
undirected, random and deterministic graphs. 
\begin{description}
\item[(DDD)] Deterministic directed graphs $\Gn=\langle V(\Gn), E_d(\Gn), A_n=(a_{n,ij})\rangle$:
\be\lbl{a-det}
a_{n,ij} =\langle W\rangle_{I_{n,i}\times I_{n,j}}:= n^2 \int_{I_{n,i}\times I_{n,j}} W(x,y)dxdy.
\ee
\item[(DDU)] If $W$ is a symmetric function, the same formula defines an undirected graph 
$\Gn=\langle V(\Gn), E(\Gn), A_n\rangle$.
\item[(RDD)] W-random graphs. Let $W: I^2\to I$ be a nonnegative measurable function.
$\Gn=G_d(n,W)$ is a directed random graph on $n$ defined as follows:
\be\lbl{Pdir}
\P\left( j\to i\right)=\langle W\rangle_{I_{n,i}\times I_{n,j}}.
\ee
\item[(RDU)] If $W$ is a symmetric function, define an undirected random graph $\Gn=G(n,W)$ as 
follows
\be\lbl{Pund}
\P\left( i \sim j\right)=\langle W\rangle_{I_{n,i}\times I_{n,j}}.
\ee
\item[(RSD)] Sparse directed W-random graph $\Gn=G_d(n,W,\alpha_n)$. Here, we assume that
$W\in L^1(I^2)$ is a nonnegative function and $1\ge \alpha_n\searrow
0$ such that $n\alpha_n\to \infty$ as $n\to\infty$. The probability 
of connection between two nodes is defined as follows
\footnote{Throughout this paper, $a\wedge b:=\min\{a,b\}$.}
 \be\lbl{Psparse-d}
\P\left( j\to i\right)=\alpha_n \langle \tilde W_n \rangle_{I_{n,i}\times I_{n,j}}, \quad 
\tilde W_n(x,y):= \alpha_n^{-1} \wedge W(x,y).
\ee
\item[(RSU)] The undirected sparse W-random graph $\Gn=G(n,W,\alpha_n)$ is defined in exactly the same way
\be\lbl{Psparse-und}
\P\left(  i\sim j\right)=\alpha_n \langle \tilde W_n \rangle_{I_{n,i}\times I_{n,j}}, 
\ee
assuming that $W$ is a symmetric nonnegative function.
\end{description}
In the KM \eqref{KM} on random graphs , we assume that $a_{n,ij}$ are
Bernoulli random variables with the probability of success defined by \eqref{Pdir}-\eqref{Psparse-und}.
For undirected graphs, we assume that $a_{n,ij}= a_{n,ji}$.
\begin{rem} 
The sequences of undirected graphs constructed above are  convergent
in the sense of convergence of dense graphs \cite{LovGraphLim12} and its generalization
to sparse random graphs of unbounded degree \cite{BCCZ}. In this paper, we will refer 
to any of the graph sequence constructed above as a convergent sequence of graphs. 
The graphon $W$ determines the asymptotic properties of each of these graph sequences.
For this reason, $W$ is called a graph limit.
\end{rem}

\begin{ex}\lbl{ex.graphs}
\begin{enumerate}
\item Sparse power law graph. 
Let $0<\beta<\gamma<1,$ $\alpha_n=n^{-\gamma}$ and
\be \lbl{Walpha}
W(x,y)=(1-\beta)^2 (xy)^{-\beta}.
\ee
Then the probability of connections in $\Gamma_n=G(n,W,\alpha_n)$ is given by
\be\lbl{Pedge-PL}
\P(i\sim j) =n^{-\gamma} \langle  n^\gamma \wedge W \rangle_{I_{n,i}\times I_{n,j}}.
\ee
The expected degree $\E \operatorname{deg}(i) = C(\beta,\gamma, n) i^{-\beta}$
for some positive constant  $C(\beta, \gamma, n)$ \cite[Lemma~2.2]{KVMed17}.
Thus, this is a power law graph. On the other hand the expected edge density 
is $O(n^{-\beta})$. Thus, $\{\Gamma_n\}$ is a sparse sequence.

If \eqref{Pedge-PL} is replaced by 
$$
\P(j \to i)= n^{-\beta} \langle  n^\beta \wedge W \rangle_{I_{n,i}\times I_{n,j}},
$$
we obtain a sequence of sparse directed graphs with power law distribution.
\item Sparse Erd\H{o}s-R{\' e}nyi graph. 
Let $\alpha_n=n^{-\gamma}, \; 0<\gamma<1$ and $W\equiv 1$. $\Gamma_n=G(n,W,\alpha_n)$
is a graph on $n$ nodes with the probability of edges being
\be\lbl{Pedge-ER}
\P(i\sim j)= n^{-\gamma}.
\ee
The expected value of the edge density in this case is $n^{-\gamma}$ and it is 
vanishing as $n\to\infty$. However, the expected degree $n^{1-\gamma}$
remains unbounded. 

If \eqref{Pedge-ER} is replaced by 
$$
\P(j \to i)=  n^{-\gamma}.
$$
we obtain a sequence of sparse directed Erd\H{o}s-R{\' e}nyi graphs.
\end{enumerate}
\end{ex}

Let $\Gamma_n=G(n,W,\alpha_n)$ be a random sparse directed graph (cf.~\textbf{(RSD)}).
The number of directed edges pointing to $i\in [n]$ is called an in-degree of $i$:
\be\lbl{in-d}
d^+_{n,i}= \sum_{j=1}^n \1_{\{j\to i\}}.
\ee
Similarly,
\be\lbl{out-d}
d^-_{n,i}= \sum_{j=1}^n \1_{\{i\to j\}}
\ee
is called an out-degree of $i\in [n]$.

From the definition of $\Gamma_n=G(n,W,\alpha_n)$, we immediately have
\begin{eqnarray}
\lbl{Ein}
\E d^+_{n,i}&=& \sum_{j=1}^n \alpha_n \langle \tilde W\rangle_{I_{n,i}\times I_{n,j}} =
\alpha_n n \int_0^1 \bar W_n (x,y)dy, \; x\in I_{n,i},\\
\lbl{Eout}
\E d^-_{n,i}&=& \sum_{j=1}^n \alpha_n \langle \tilde W\rangle_{I_{n,j}\times I_{n,i}} =
\alpha_n n \int_0^1 \bar W_n (y,x)dy, \; x\in I_{n,i},
\end{eqnarray}
where $\bar W_n=\sum_{i,j=1}^n  \langle \tilde W\rangle_{I_{n,i}\times I_{n,j}} \1_{I_{n,i}\times I_{n,j}}.$

The following assumptions will be needed below:
\begin{description}
\item[(W-1)]
$$
\sup_{n\in\N}\sup_{y\in I} \int_0^1 W_n(x,y)dx =:W_1<\infty,
$$
\item[(W-2)]
$$
\sup_{i\in [n]}\sup_{x\in I} \int_0^1  W_n(x,y)dy =:W_2<\infty,
$$
\end{description}
where
\be\lbl{step-W}
W_n(x,y)=\sum_{i,j=1}^n \langle W\rangle_{I_{n,i}\times I_{n,j}} \1_{I_{n,i}\times I_{n,j}}(x,y).
\ee.

Conditions \textbf{(W-1)} and \textbf{(W-2)} clearly imply
\be\lbl{W1-W2}
\sup_{n\in\N}\sup_{y\in I} \int_0^1 \bar W_n(x,y)dx \le W_1, \quad
\sup_{i\in [n]}\sup_{x\in I} \int_0^1 \bar W_n(x,y)dy \le W_2.
\ee

\begin{rem}\lbl{rem.degree}
Conditions \eqref{W1-W2} apply to undirected random graphs 
as well. In the undirected case, the two conditions are equivalent, since $W$ is a symmetric function.
Furthermore, by setting $\alpha_n\equiv 1$ and restricting to $W\in L^\infty(I^2)$, both conditions 
apply to directed and undirected dense W-random graphs. With these conventions, below  it
will be always assumed that conditions \eqref{W1-W2} hold for any of
the above types of graphs.
\end{rem}

Conditions \eqref{W1-W2} mean that the (in-) and out-degree of 
any node in $\Gamma_n$ are $O(\alpha_n n)$. The uniformity here is the key.
Both conditions clearly hold for all dense graphs 
(i.e., $W\in L^\infty(I^2)$) and many sparse graphs. For instance, sparse Erd\H{o}s-R{\' e}nyi 
and small-world graphs satisfy this condition. However, not every $\Gamma_n=G_d(n,W,\alpha_n)$ 
satisfies \eqref{W1-W2}. For instance, the power law graph defined in 
Example~\ref{ex.graphs} does not satisfy \eqref{W1-W2}. At the end of the next section, we 
show that the KM on the power law graphs, after a suitable rescaling of the coupling term, can 
still be analyzed with the techniques of this paper.

The nonnegativity assumption $W\ge 0$ is used for convenience and can be dropped. 
Indeed, writing $W=W^+-W^-$, assume that positive and negative parts of $W$,  $W^+$
and $W^-,$ satisfy \textbf{(W-1)} and \textbf{(W-2)}\footnote{It is
  sufficient to assume: 
$
\esssup_{x\in I} \int_I |W(x,y)| dy \le W_1,\quad
\esssup_{y\in I} \int_I |W(x,y)| dx \le W_2.
$
}. Then one can define graphs on $n$ nodes,
$\Gn^+$ and $\Gn^-$, whose edge sets are defined using the graphons $W^+$ and $W^-$
respectively. Thus, the original model can be rewritten as
\be\lbl{signedKM}
\dot u_{n,i} = f(u_{n,i},t) +(n\alpha_n)^{-1}\left(\sum_{j=1}^n a^+_{n,ij} D(u_{n,j}-u_{n,i})-
\sum_{k=1}^n a^-_{n,ik} D(u_{n,k}-u_{n,i})\right), \; i\in [n],
\ee
where $(a^+_{n,ij})$ and $(a^-_{n,ij})$ are weighted adjacency matrices of $\Gn^+$ and $\Gn^-$.
The derivation and analysis of the continuum limit for \eqref{KM} with nonnegative $W$
translates verbatim for \eqref{signedKM}. To simplify presentation, we restrict to
the case of nonnegative $W$.

\section{The main result} \lbl{sec.main}
\setcounter{equation}{0}

Having defined the discrete model \eqref{KM}, \eqref{KM-ic}, we now present the 
main result of this work. Our goal is to describe the limiting behavior of the coupled system as 
$n\to\infty$. Specifically, we are going to compare the solutions of the discrete model 
\eqref{KM}, \eqref{KM-ic} for large $n$ with the solution of the IVP for the continuum model
\begin{eqnarray}\lbl{clim}
\p_t u(t,x) &=& f(u,t) +\int_I W(x,y) D\left(u(t,y)-u(t,x)\right) dy, \quad x\in I,\\
\lbl{clim-ic}
u(0,x)&=& g(x).
\end{eqnarray}
For simplicity of exposition, we 
state our main result for the KM on a sparse directed graph. Clearly, the statement of the 
theorem translates easily to the 
KM on sparse undirected graphs, as well as both directed and undirected dense random graphs,
as explained in Remark~\ref{rem.degree}.

Below, we use the bold font to denote  $X$-valued functions. In particular,
$\mathbf{u}(t)$ stands for the map $t\mapsto u(t,\cdot)\in X$. Further, given the solution
of the IVP  for the  KM \eqref{KM}, \eqref{KM-ic} 
$u_n(t)=\left(u_{n,1}(t), u_{n,2}(t),\dots, u_{n,n}(t)\right),$ we define 
\be\lbl{un-step}
u_n(t,x) =\sum_{i=1}^n u_{n,i}(t) \phi_{n,i} (x),
\ee
where $\phi_{n,i}(x)=\1_{I_{n,i}}(x)$ is the characteristic function of $I_{n,i}$, $i\in [n]$.
The corresponding vector-valued function is denoted by $\mathbf{u}_n(t).$

\begin{thm}\lbl{thm.main}
Let $u_n(t)=\left(u_{n,1}(t), u_{n,2}(t),\dots, u_{n,n}(t)\right)$ be  the solution of the IVP for the 
KM \eqref{KM}, \eqref{KM-ic} on $\Gn=G_d(n,W, n^{-\gamma}),$ $0<\gamma<0.5,$ 
with nonnegative $W\in L^2(I^2)$ satisfying \textbf{(W-1)} and \textbf{(W-2)} and subject 
to the initial condition \eqref{KM-ic} with
$$
u^0_{n,i}=\langle g\rangle_{I_{n,i}},\quad i\in [n],
$$
and $g\in L^2(I).$ 

Then for any $T>0$,
$$
\lim_{n\to\infty} \| \mathbf{u}_n-\mathbf{u} \|_{C(0,T; L^2(I))}=0\quad \mbox{a.s.},
$$
where $\bu(t)$ is the solution of the IVP for the continuum limit \eqref{clim}, \eqref{clim-ic}
and $\bu_n(t)$ is defined by \eqref{un-step}.
\end{thm}

Theorem~\ref{thm.main} establishes convergence of the discrete models on graphs to the continuum
limit under the minimal assumptions on $W$. We only ask that the graphon $W\in L^2(I^2)$ satisfies
technical conditions \textbf{(W-1)} and \textbf{(W-2)}. This allows us to treat the KM on a variety 
of graphs in a uniform fashion. In particular, Theorem~\ref{thm.main} contains as special cases
convergence of the KM on dense deterministic and random graphs analyzed in 
\cite{Med14a, Med14b}, as well as convergence of the KM on sparse graphs considered in 
\cite{KVMed17}.
In the case of the KM on random graphs, in this paper the convergence is proved in the almost 
sure sense compared to the convergence in probability in \cite{Med14b, KVMed17}. In addition,
the setting of this paper includes the KM on directed graphs, while the even symmetry of $W$
was used in certain arguments in \cite{KVMed17}. All in all, the main result of this paper shows
convergence of the KM to the continuum limit in the stronger sense and for a more general
class of graphs than in the previous work on this subject.

The proof of Theorem~\ref{thm.main} follows the scheme developed in \cite{Med14b, KVMed17}.
The first step of the proof is estimating proximity between the solution on the IVP \eqref{KM}, 
\eqref{KM-ic} and that for the averaged equation:
\begin{eqnarray}
\lbl{aveKM}
\dot v_{n,i} &=& f(v_{n,i},t) + (n\alpha_n)^{-1} \sum_{j=1}^n \bar W_{n,ij} D(v_{n,j}-v_{n,i}),
\quad i\in [n],\\
\lbl{aveKM-ic}
v_{n,i}(0)&=&u_{n,i}^0. 
\end{eqnarray}
Here, we replaced $a_{n,ij}, i,j \in [n],$ with their expected values. In Theorem~\ref{thm.ave} 
below, we prove that the 
solutions of the original and averaged models with probability $1$ become closer and closer 
in the appropriate norm for increasing values of $n$. On the other hand, \eqref{aveKM} has the form 
of a cartesian discretization of the continuum limit \eqref{clim}. Thus, the second step in the 
proof is to show that the averaged model approximates the nonlocal equation \eqref{clim}.
This step is accomplished by showing that the averaged model is asymptotically equivalent to a 
Galerkin approximation of the continuum model. Here, we employ the corresponding argument 
from \cite{KVMed17}. 

The new challenges in implementing this plan are twofold. On the one hand, we significantly relaxed
the assumptions on $W$, compared to $W\in L^\infty(I^2)$ used in \cite{Med14a, Med14b}. On 
the other hand, we consider the nonlinear interaction function $D$ compared to the linear diffusion in
\cite{KVMed17}. To overcome these problems we refined our techniques. This includes 
the use of the concentration inequalities in the proof of Theorem~\ref{thm.ave} to obtain finer
estimates on the solutions of the averaged model, and the revision of the Galerkin scheme from
\cite{KVMed17} so that it covers the model with nonlinear interaction function $D$. While the overall
approach remains the same as in \cite{Med14b, KVMed17}, the analysis in the present paper 
reflects a better understanding of the method of the proof of convergence to the continuum limit, and 
fuller reveals its potential.

\section{Averaging}\lbl{sec.average}
\setcounter{equation}{0}

For KM on random graphs, the key step in the derivation of the continuum limit is the 
averaging procedure, when a stochastic model is approximated by a deterministic (averaged)
system. In this section, we focus on the justification of the averaging.

Throughout this section, we consider the KM on random graphs (cf.~\textbf{RDD},
\textbf{RDU}, \textbf{RSD}, \textbf{RSU}). Without loss of generality, we consider the KM
on a random sparse directed graph $\Gamma_n=G_d(n,W,\alpha_n),$ as it represents the 
most general case\footnote{For the KM on an undirected graph, assume, in addition,
that $W$ is symmetric. In the dense case, restrict to $0\le W\le 1$ and set $\alpha_n\equiv 1$.}.

For convenience, we rewrite the KM on $\Gamma_n=G_d(n,W,\alpha_n):$
\be\lbl{rKM}
\dot u_{n,i} = f( u_{n,i}, t) +{1\over \alpha_n n}\sum_{j=1}^n a_{n,ij} 
D(u_{n,j}- u_{n,i}),\quad i\in [n],
\ee

Taking the expected value of the right-hand side of \eqref{rKM} on $\Gn$
$$
\E a_{n,ij}= \P( j\to i) =\alpha_n \langle \tilde W_n\rangle_{I_{n,i}\times I_{n,j}},
$$
we arrive at the following
averaged model
\be\lbl{aKM}
\dot v_{n,i} = f(v_{n,i}, t) +{1\over n}\sum_{j=1}^n \bar W_{n,ij} D(v_{n,j}-v_{n,i}),\quad
i\in [n],
\ee
where 
$
\bar W_{n,ij}:=\langle \tilde W_n\rangle_{I_{n,i}\times I_{n,j}}.
$

To compare the solutions of the IVPs for the original and the averaged KMs, we adopt the 
discrete $L^2$-norm:
\be\lbl{dL2}
\| u_n-v_n \|_{2,n} =\left( n^{-1} \sum_{j=1}^n (u_{n,i}-v_{n,i})^2\right)^{1/2}.
\ee

\begin{thm}\lbl{thm.ave}
Let nonnegative $W\in L^1(I^2)$ satisfy \textbf{(W-1)}, \textbf{(W-2)},
and $\alpha_n=n^{-\gamma}$, $\gamma\in (0, 0.5),$ and
\be\lbl{def-L}
L=L_f+L_D\left(2+{3\over 2} W_1+{1\over 2}W_2\right)+{1\over 2}.
\ee
 Then for solutions of the 
original and averaged equations  \eqref{rKM} and \eqref{aKM} subject to the 
same initial conditions and any $T\le C\ln n$, $0\le C< (1-2\gamma)L^{-1}$,
 we have
\be\lbl{ave.statement}
\lim_{n\to\infty} \sup_{t\in [0,T]} \| u_n -v_n\|_{2,n} =0\quad \mbox{almost surely (a.s.)}.
\ee
\end{thm}

\begin{proof}
Recall that $f(u,t)$ and $D$ are Lipschitz continuous function in $u$ with Lipschitz 
constants $L_f$ and $L_D$ respectively. In addition, $f(u,t)$ is a continuous function
of $t$ and $D(u)$ is $2\pi$--periodic function satisfying \eqref{boundD}.
 
Further,  $a_{n,ij},$ are Bernoulli random variables
\be\lbl{Bern}
\P(a_{n,ij}=1)=\alpha_n \bar W_{n,ij}.
\ee

Denote $\psi_{n,i}:=v_{n,i}- u_{n,i}.$ 
By subtracting (\ref{rKM}) from (\ref{aKM}), 
multiplying the result by $n^{-1}\psi_{n,i},$ and summing over $i\in [n]$, we obtain
\be \lbl{subtract}
\begin{split}
{1\over 2} {d\over dt} \|\psi_n\|^2_{2,n} &= \underbrace{
n^{-1}\sum_{i=1}^n \left(f(v_{n,i}, t)-f(u_{n,i}, t)\right) \psi_{n,i}}_{I_1}\\
&+
\underbrace{ n^{-2} \alpha_n^{-1}\sum_{i,j=1}^n
\left(\alpha_n W_{n,ij}-a_{n,ij}\right)
D(v_{n,j}-v_{n,i}) \psi_{n,i}}_{I_2}\\
&+
\underbrace{n^{-2} \alpha_n^{-1} \sum_{i,j=1}^n a_{n,ij} 
\left[ D(v_{n,j}-v_{n,i}) - D(u_{n,j}-u_{n,i})\right]\psi_{n,i}}_{I_3}\\
&=:I_1+I_2+I_3,
\end{split}
\ee
where $\| \cdot\|^2_{2,n}$ is the discrete $L^2$-norm (cf.~\eqref{dL2}).

Using Lipschitz continuity of $f$ in $u$, we  have
\be\lbl{I_1}
|I_1|\le L_f \|\psi_{n}\|_{2,n}^2.
\ee
Using Lipschitz continuity of $D$ and the triangle inequality, we 
have
\be\lbl{I_3}
\begin{split}
|I_3|& \le L_D n^{-2} \alpha_n^{-1} \sum_{i,j=1}^n a_{n,ij} \left( |\psi_{n,i}|+|\psi_{n,j}|\right) |\psi_{n,i}|\\
& \le L_D n^{-2} \alpha_n^{-1} 
\left({3\over 2} \sum_{i,j=1}^n  a_{n,ij} \psi_{n,i}^2 +{1\over 2} \sum_{i,j=1}^n  a_{n,ij} \psi_{n,j}^2
\right).
\end{split}
\ee

Choose $0<\delta<1-2\gamma$ and denote
\begin{eqnarray}
\lbl{Ani}
A_{n,i}&=&\left\{ S_{n,i}  \ge \alpha_n\sum_{j=1}^n\bar  W_{n,ij} + n^{{1+\delta\over 2}}\right\}, 
\quad n\in \N, \; i\in [n],\\
\lbl{An}
A_n&=& \bigcup_{i=1}^n A_{n,i},\quad n\in \N,
\end{eqnarray}
where
\be\lbl{def-Sni}
S_{n,i}=\sum_{j=1}^n a_{n,ij},\quad i\in [n].
\ee

Noting $0\le a_{n,ij}\le 1,\; \E S_{n,i}=\alpha_n \sum_{j=1}^n\bar W_{n,ij}$, we apply 
Chernoff-Hoeffding inequality\footnote{ 
Here and below, we are using 
$$
 \P\left(\sum_{i=1}^N X_i\ge \sum_{i=1}^N \E X_i +t\right)\le
e^{{-2t^2\over \sum_{i=1}^N (b_i-a_i)^2}},
\;\mbox{and}\;
\P\left(\left|\sum_{i=1}^N X_i - \sum_{i=1}^N \E X_i\right| \ge t\right)\le
2e^{{-2t^2\over \sum_{i=1}^N (b_i-a_i)^2}},
$$
which hold for collectively independent random variables 
$a_i\le X_i\le b_i,\; i\in [N]$ \cite{Hoeffding}.
} 
to bound
\be\lbl{Chernoff}
\P \left(A_{n,i}\right)\le  e^{-2 n^\delta}.
\ee
By the union bound,
\be\lbl{bound-An}
\P\left(A_n\right) \le ne^{-2 n^\delta}.
\ee
By Borel-Cantelli lemma, with probability $1$ there exists $N_\delta\in \N$ such that 
\be\lbl{byBC}
S_{n,i} <\alpha_n\sum_{j=1}^n\bar  W_{n,ij}+ n^{{1+\delta\over 2}}, 
\ee
for all $n\ge N_\delta$ and $i\in [n]$.
Below, we restrict to the subset of probability $1,$ where \eqref{byBC} holds.

With \eqref{byBC} in hand, we return to bounding the right--hand side of \eqref{I_3}
\be\lbl{uh-1}
\begin{split}
n^{-2} \alpha_n^{-1} \sum_{i,j=1}^n  a_{n,ij} \psi_{n,i}^2& \le 
n^{-1}\sum_{i=1}^n \left[ n^{{-1+\delta\over 2}}\alpha_n^{-1} +
n^{-1}\sum_{j=1}^n \bar  W_{n,ij}\right] \psi_{n,i}^2\\
\le \left(1+ W_1\right) \|\psi\|_{n,2}^2,
\end{split}
\ee
where we used $n^{{-1+\delta\over 2}}\alpha_n^{-1}=n^{{-1 +2\gamma+\delta\over 2}}\le 1$
and the definition of $W_1$ (cf.~ \eqref{W1-W2}).

Similarly,
\be\lbl{uh-2}
n^{-2} \alpha_n^{-1} \sum_{i,j=1}^n  a_{n,ij} \psi_{n,j}^2
\le \left(1+ W_2\right) \|\psi\|_{n,2}^2.
\ee
By plugging \eqref{uh-1} and \eqref{uh-2} into \eqref{I_3}, we have
\be\lbl{finish-I_3}
|I_3| \le L_D \left( 2 +{3\over 2} W_1 +{1\over 2} W_2 \right)
\|\psi\|_{n,2}^2.
\ee

It remains to bound $I_2$. To this end,  we will need the following definitions:
\begin{eqnarray*}
Z_{n,i}(t)   &=& n^{-1}\sum_{j=1}^n b_{n,ij}(t)\eta_{n,ij},\\
b_{n,ij}(t) &=& D\left(v_{n,j}(t)-v_{n,i}(t)\right),\\
\eta_{n,ij} &=&a_{n,ij} -\alpha_n\bar W_{n,ij},
\end{eqnarray*}
and $Z_n=(Z_{n,1}, Z_{n,2},\dots, Z_{n,n})$.
With these definitions in hand, we estimate $I_2$ as follows:
\be\lbl{est-I-1}
|I_2|= |n^{-1}\alpha_n^{-1}
\sum_{i=1}^n Z_{n,i}\psi_{n,i}|\le 2^{-1}\alpha_n^{-2} \|Z_n\|_{2,n}^2 +
2^{-1} \|\psi_n\|^2_{2,n}.
\ee

The combination of  \eqref{subtract}, \eqref{I_1}, \eqref{finish-I_3} and \eqref{est-I-1} yields
\be\lbl{pre-G-1}
{d\over dt} \|\psi_n(t)\|_{2,n}^2 \le L\|\psi_n(t)\|_{2,n}^2 +
{1\over \alpha_n^2} \|Z_n(t)\|_{2,n}^2,
\ee
where $L$ is defined in \eqref{def-L}.

Using the Gronwall's inequality, we have
$$
\|\psi_n(t)\|_{2,n}^2\le  \alpha_n^{-2}e^{ Lt}\int_0^\tau e^{-Ls}\|Z_n(s)\|_{2,n}^2 ds.
$$
and
\be\lbl{Gron}
\sup_{t\in [0, T]} \|\psi_n(t)\|_{2,n}^2\le \alpha_n^{-2} e^{ L T}
\int_0^\infty e^{-Ls}\|Z_n(s)\|_{2,n}^2 ds.
\ee

Our next goal is to estimate $\int_0^\infty e^{-Ls}\|Z_n(s)\|_{2,n}^2 ds$.
To this end, note
\begin{eqnarray} \lbl{mean-0}
\E \eta_{n,ij} &=&  \E (a_{n,ij}-\alpha_n\bar W_{n,ij}) =0,\\
\lbl{eta2}
\E\eta_{n,ij}^2 &=&\E (a_{n,ij}-\alpha_n\bar W_{n,ij})^2= \alpha_n\bar W_{n,ij} -(\alpha_n\bar W_{n,ij})^2\le 1.
\end{eqnarray}

Further, 
\be\lbl{Zni2}
\int_0^\infty e^{-Ls} Z_{n,i}(s)^2 ds= n^{-2} \sum_{k,l=1}^n c_{n,ikl} \eta_{n,ik}\eta_{n,il}, 
\ee
where
\be\lbl{def-cnjikl}
c_{n,ikl}=\int_0^\infty e^{-Ls} b_{n,ik}(s) b_{nil}(s) ds \quad\mbox{and}\quad 
|c_{n,ikl}|\le L^{-1}. 
\ee
Further, from \eqref{Zni2} and \eqref{def-cnjikl}, we have
\be\lbl{Zni2-2}
\int_0^\infty e^{-Ls}  \|Z_{n}(s)\|_{2,n}^2 ds= 
n^{-3} \sum_{i,k,l=1}^n c_{n,ikl} \eta_{n,ik}\eta_{n,il}.
\ee

Our final goal is to bound the sum on the right--hand side of \eqref{Zni2-2}. To this end,
we write 
\be\lbl{split-in-2}
\sum_{i,k,l=1}^n c_{n,ikl} \eta_{n,ik}\eta_{n,il}
=\sum_{i,k=1}^n c_{n,ikk} \eta_{n,ik}^2 +
2\sum_{i=1}^n \sum_{1\le l<k\le n}  c_{n,ikl} \eta_{n,ik}\eta_{n,il}.
\ee
Both sums on the right--hand side of \eqref{split-in-2} are formed of independent
bounded random random variables.  By Chernoff-Hoeffding 
inequality, for an arbitrary $\delta>0$, we have
\begin{eqnarray}
\lbl{Chernoff-H-1}
 \P\left( \sum_{i,k=1}^n c_{n,ikk} \eta_{n,ik}^2\ge \sum_{i,k=1}^n c_{n,ikk} \E\eta_{n,ik}^2+n^2\right)
&\le& e^{-n^2L^2},\\
\lbl{Chernoff-H-2}
\P\left(\left|\sum_{i=1}^n \sum_{1\le l<k\le n}  c_{n,ikl} \eta_{n,ik}\eta_{n,il} \right|\ge 
n^{{3\over 2}+\delta} \right) &\le & 2 e^{-n^{2\delta} L^2},
\end{eqnarray}
where we used the bound on $c_{n,ikl}$ (see \eqref{def-cnjikl}).
By Borel-Cantelli lemma, we now have
\begin{eqnarray}
\lbl{BC-1}
 \sum_{i,k=1}^n c_{n,ikk} \eta_{n,ik}^2\le \sum_{i,k=1}^n c_{n,ikk} 
\E\eta_{n,ik}^2+n^2
&<& (L^{-1}+1)n^2,\\
\lbl{BC-2}
\left|\sum_{i=1}^n \sum_{1\le l<k\le n}  c_{n,ikl} \eta_{n,ik}\eta_{n,il} \right|
&< & n^{{3\over 2}+\delta}, 
\end{eqnarray}
for sufficiently large $n$ a.s.. Plugging in these bounds into \eqref{split-in-2}
and \eqref{Zni2-2}, we obtain
\be\lbl{bound-Zn}
\alpha_n^{-2} \int_0^\infty e^{-Ls} \| Z_n(s)\|_{2,n}^2 ds \le
\alpha_n^{-2} \left( (L^{-1}+1)n^{-1}+ 2n^{{-3\over 2}+\delta} \right) \le C_1 \alpha_n^{-2} n^{-1}
\ee
for some $C_1\ge 0$ a. s..

Using \eqref{bound-Zn}, from \eqref{Gron} we have
\be\lbl{final-bound}
\sup_{t\in [0, T]} \|\psi_n(t)\|_{2,n}^2\le C_1 e^{ L T} \alpha_n^{-2} n^{-1}.
\ee
For $\alpha_n=n^{-\gamma}, \; 0<\gamma<{1\over 2}$ the right--hand
side of \eqref{final-bound} tends to zero on the time interval
with $T \le C \ln n$ for any $0<C<{1-2\gamma\over L}.$
\end{proof}

If we restrict to finite time intervals then \eqref{final-bound} yields
the rate of convergence estimate.
\begin{cor} \lbl{cor.finite-T} For fixed $T>0,$ we have
\be\lbl{finite-T}
\lim_{n\to\infty} n^{{1\over 2}-\gamma-\delta}\sup_{t\in [0, T]} 
\|\psi_n(t)\|_{2,n}=0\quad\mbox{a.s.},
\ee
where $0<\delta<{1\over 2}-\gamma$ is arbitrary.
\end{cor}
\begin{rem}\lbl{rem.dense}
Theorem~\ref{thm.ave} and Corollary~\ref{cor.finite-T}
clearly apply to the KM on undirected sparse graphs. Furthermore, by setting $\gamma=0$, 
these results translate to the KM on dense W--random graphs.
\end{rem}

\begin{rem}
As we pointed out earlier, not every sparse random graph defined in (\textbf{RSD}, \textbf{RSU})
meets \eqref{W1-W2}. However, the averaging can still be justified for the 
KM on such graphs if the original model is suitably rescaled. For simplicity, we explain the new
scaling for the KM on undirected graphs. 
 
Let $\Gn=G(n,W,\alpha_n),$ where $W\in L^1(I^2)$ is a symmetric nonnegative function and 
$\alpha_n\searrow 0,$ $\alpha_nn\to \infty$ as before. Consider
\be\lbl{dKM}
\dot u_{n,i} = f(u_{n,i},t) + d_{n,i}^{-1} \sum_{j=1}^n a_{n,ij} D(u_{n,j}-u_{n,i}),\quad i\in [n],
\ee
where $d_{n,i}:=d^+_{n,i}=d^-_{n,i}$ is a degree of node $i\in [n].$ We claim that the conclusion
of Theorem~\ref{thm.ave} holds for the rescaled model \eqref{dKM} for any nonnegative 
symmetric $W\in L^1(I^2)$. Indeed, the averaged system in this case takes the following form
\be\lbl{adKM}
\dot v_{n,i} = f(v_{n,i},t) + n^{-1} \sum_{j=1}^n U_{n,ij} D(v_{n,j}-v_{n,i}),\quad i\in [n],
\ee
where 
\be\lbl{defU}
U_{n,ij}={\bar W_{n,ij}\over n^{-1}\sum_{k=1}^n \bar W_{n,ki}}, \quad (i,j)\in [n]^2.
\ee
Using $\bar W_{n,ij}=\bar W_{n,ji}$ and \eqref{defU}, we have
$$
n^{-1}\sum_{k=1}^n U_{n,kj} =n^{-1}\sum_{k=1}^n U_{n,ik}=1\quad \forall i,j\in [n].
$$ 
Thus, the bounds in \eqref{uh-1} and \eqref{uh-2} 
hold with $W_1=W_2=1$. The rest of the proof remains unchanged.
\end{rem}

\section{The continuum limit}\lbl{sec.continuum}
\setcounter{equation}{0}


We now turn to the IVP for the continuum model \eqref{clim}, \eqref{clim-ic}.
The solution of the IVP \eqref{clim}, \eqref{clim-ic} will be understood in a weak sense. 
Specifically,
let  $T>0$ and $X$ stand for $L^2(I)$.  Denote
\be\lbl{def-K}
K(u(t, \cdot)):= \int_I W(\cdot,y) D\left(u(t,y)-u(\cdot, t)\right)dy.
\ee
$K$ is viewed as an operator on $L^2(I)$.

\begin{df}\lbl{weak-solution} \cite{KVMed17}
$\mathbf{u}\in H^1(0,T; X)$ is called a weak solution of the IVP
(\ref{clim}), (\ref{clim-ic})
on $[0,T]$ if
\be\lbl{weak-sol}
\left( \mathbf{u^\prime}(t)- K(\mathbf{u}(t)) -f(\mathbf{u}(t),t), \mathbf{v}\right)=0\quad
\forall \mathbf{v}\in X
\ee
almost everywhere (a.e.) on $[0,T]$
and $\mathbf{u}(0)=g$.
\end{df}

The averaged equation \eqref{aKM} can be rewritten as a diffusion equation on $[0,1]$ for
the step function
\be\lbl{step-v}
v_n(t,x)  = \sum_{i=1}^n v_{n,i}(t) \phi_{n,i}(x),
\ee
where $\phi_{n,i}(x), i\in [n],$ is the step function defined right
after \eqref{un-step}.
Specifically, the IVP for \eqref{aKM} has the following form
\begin{eqnarray}
\lbl{dvn} 
\p_t v_n(t,x)& =& f(v_n(t,x), t) + \int_I \bar W_n(x,y) D\left(v_n(t,y)-v_n(t,x)\right) dy,\\
\lbl{dvn-ic}
v_n(0,x)& = & g_n(x),
\end{eqnarray}
where 
\begin{eqnarray}
\lbl{step-g}
g_n(x)&=&\sum_{i=1}^n g_{n,i} \phi_{n,i}(x),\quad g_{n,i}=\langle g\rangle_{I_{n,i}} :=
n\int_{I_{n,i}} g(x)dx,\\
\lbl{step-W}
\bar W_n(x,y) & = &\sum_{i,j=1}^n \bar W_{n,ij} \phi_{n,i}(x)\phi_{n,j}(y).
\end{eqnarray}


The following theorem establishes the continuum limit for the IVP for the averaged equation
\eqref{dvn}, \eqref{dvn-ic}.

\begin{thm}\lbl{thm.clim}
Let $W\in L^2(I^2)$ satisfy \textbf{(W-1)}, \textbf{(W-2)} and $g\in L^2(I)$. 
Recall that $f(u,t)$ and $D(u)$ are Lipschitz 
continuous functions in $u$. In addition, $f(u,t)$ is a continuous function of $t$.

For $T>0$, there is a unique weak solution of the IVP \eqref{clim}, \eqref{clim-ic}.
Moreover,
$$
\lim_{n\to\infty} \| \mathbf{v}_n-\mathbf{u} \|_{C(0,T; L^2(I))}=0,
$$
where $\bu(t)$ is the solution of the IVP for the continuum limit \eqref{clim}, \eqref{clim-ic}
and $\mathbf{v}_n(t)=v_n(t,\cdot)$ is the solution of the IVP \eqref{dvn}, \eqref{dvn-ic}.
\end{thm}

Theorem~\ref{thm.clim} combined with Theorem~\ref{thm.ave} implies Theorem~\ref{thm.main},
which provides a rigorous justification for the continuum limit of the KM on sparse graphs.

\section{Proof of Theorem~\ref{thm.clim}}\lbl{sec.converge}
\setcounter{equation}{0}

In this section, we prove existence and uniqueness of solution of the IVP 
\eqref{clim}, \eqref{clim-ic}. We show that the solutions of the finite-dimensional
Galerkin problems converge to the unique weak solution of the IVP \eqref{clim}, 
\eqref{clim-ic}. The Galerkin problem, in turn, is very close to the IVP for the averaged
equation \eqref{dvn}, \eqref{dvn-ic}. Thus, convergence of Galerkin problems 
to the continuum limit \eqref{clim}, the main result of this section, almost 
immediately implies Theorem~\ref{thm.main}.
 
\subsection{Galerkin problems}

Recall
\be\lbl{def-phi}
\phi_{n,i}(x)=\1_{I_{n,i}}(x)=\left\{ \begin{array}{ll}
1, & x\in I_{n,i},\\
0, & x \not\in I_{n,i},
\end{array}\right.\; i\in [n],
\ee
and consider a finite dimensional subspace of $X$,
$X_n=\operatorname{span}\{\phi_{n,1}, \phi_{n,2}, \dots, \phi_{n,n}\}$.
We now consider a Galerkin approximation of the continuum problem 
\eqref{clim}, \eqref{clim-ic}:
\be\lbl{project} 
\left(\mathbf{u}_n^\prime(t)- K(\mathbf{u}_n(t))-f(\mathbf{u}_n(t)),
\phi \right)=0\quad \forall \mathbf{\phi}\in X_n, 
\ee
where
\be\lbl{project-ic} 
\mathbf{u}_n(0)=\sum_{i=1}^n g_{n,i} \phi_{ni}, 
\ee 

By plugging 
\be\lbl{Gsol}
\mathbf{u}_n(t)=\sum_{i=1}^n u_{n,i}(t)\mathbf{\phi}_{n,i}. 
\ee 
into \eqref{project} with $\phi:=\phi_{n,i}, \; i\in [n]$, we obtain the 
following system of equations for the coefficients $u_{n,i}(t)$:
\begin{eqnarray}\lbl{gKM}
\dot u_{n,i} &=& f(u_{n,i}, t) +{1\over n}\sum_{j=1}^n W_{n,ij} D(u_{n,j}-u_{n,i}),\quad i\in [n],\\
\lbl{gKM-ic}
u_{n,i}(0)&=& g_{n,i},
\end{eqnarray}
where $W_{n,ij}=\langle W\rangle_{I_{n,i}\times I_{n,j}}$ (cf.~\eqref{step-W}).

The following lemma shows wellposedness of the IVP for \eqref{clim}, \eqref{clim-ic}.
It also justifies using \eqref{clim} as the continuum limit for the KM \eqref{gKM}  on dense graphs 
(\textbf{DDD}, \textbf{DDU},\textbf{RDD}, \textbf{RDU}).
\begin{lem}\lbl{lem.converge} 
There is a unique weak solution of \eqref{clim}, \eqref{clim-ic},
$\mathbf{u}\in H^1(0,T;X).$
The solutions of the Galerkin problems \eqref{project},
\eqref{project-ic}, $\mathbf{u}_n$ converge to
$\mathbf{u}$ in the $L^2(0,T; X)$ norm as $n\to\infty$.
\end{lem}
\begin{rem} 
Under additional condition $\int_I W(x,y)dy =1$ a.e. $x\in I$,
there exists  a unique strong solution of \eqref{clim}, \eqref{clim-ic},
$\mathbf{u}\in C^1(0,T;X)$ (cf.~\cite[Theorem~3.1]{KVMed17}).
\end{rem}

We rewrite \eqref{gKM}, \eqref{gKM-ic} as a nonlocal diffusion equation
\begin{eqnarray}\lbl{dun}
\p_t u_n(t,x)&=&f(u_n(t,x), t) + \int_I W_n(x,y) D\left(u_n(t,y)-u_n(t,x)\right) dy,\\
\lbl{dun-ic}
 u_n(0,x) &=& g_n(x),
\end{eqnarray}
where 
\be\lbl{step-u}
u_n(t,x)  = \sum_{i=1}^n u_{n,i}(t) \phi_{n,i}(x)
\ee
and $W_n$ is defined in \eqref{step-W}.

Throughout the remainder of this paper, $\|\cdot\|$ stands for the norm in $X=L^2(I)$.
Equation \eqref{step-u} establishes one-to-one correspondence between $u_n(t,\cdot)\in C(\R,X_n)$
and $u_n(t)=(u_{n,1}(t), u_{n,2}(t),\dots, u_{n,n}(t))\in C(\R,\R^n).$ Moreover, 
$\| u_n(t,\cdot)\|=\|u_n(t)\|_{2,n}$.

\begin{lem}\lbl{lem.Gexist}
For every $n\in N$, there exists a unique solution of the discrete problem 
\eqref{gKM}, \eqref{gKM-ic} defined on $\R$.
\end{lem}
\begin{proof}
Denote the right-hand side of \eqref{dun} by $R_n(u_n(t,x))$. 
We show that $R_n$ is Lipschitz continuous with Lipschitz constant independent on $n$.
From \eqref{dun}, using triangle inequality, for $\bu_n, \bv_n\in C(0,T;X_n)$ we have
\begin{eqnarray}\lbl{use-triangle}
\|R_n(u_n(t,\cdot),t)-R_n(v_n(t,\cdot),t)\|& \le &\|f(u_n(t,\cdot),t)-f(v_n(t,\cdot),t)\|\\
\nonumber
&+& \left\|\int_I W_n(\cdot,y) \left[ D(u_n(t,y))-u_n(t,\cdot))- D(v_n(t,y))-v_n(t,\cdot))\right] dy\right\|\\
\nonumber
&=:& R^{(1)}+R^{(2)}.
\end{eqnarray}
Using Lipschitz continuity of $f$, we  have\footnote{Recall that $L_D$ and $L_f$ 
are Lipschitz constants of $D(u)$ and $f(u,t)$ as functions of $u$.}
\be\lbl{R-1}
R^{(1)}\le L_f \|u_n(t,\cdot)- v_n(t,\cdot) \|.
\ee
Using Lipschitz continuity of $D$ and the triangle inequality, we have
\be\lbl{R-2}
\begin{split}
R^{(2)} &\le  L_D \left\| \int_I W_n(\cdot, y)\left( \left| u_n(t,y)- v_n(t,y) \right| +
\left| u_n(t,\cdot)- v_n(t,\cdot) \right|\right)dy \right\|\\
&\le 
L_D\left( \left\| \int_I W_n(\cdot, y) \left| u_n(t,y)- v_n(t,y) \right| dy \right\|+
\left\| \int_I W_n (\cdot, y) dy \left| u_n(t,\cdot)- v_n(t,\cdot) \right| \right\|\right)\\
&=:  L_D\left(R^{(3)}+R^{(4)}\right).
\end{split}
\ee
By the Cauchy-Schwarz inequality, 
\be\lbl{R-3a}
R^{(3)} \le \left\| W_n \right\|_{L^2(I^2)} \|u_n(t,\cdot)- v_n(t,\cdot) \|.
\ee
Since $W_n$ is an $L^2$-projection of $W$ onto $X_n\otimes X_n$,
$\|W_n\|_{L^2(I^2)}\le \|W\|_{L^2(I^2)}$. Thus, \eqref{R-3a} yields
\be\lbl{R-3}
R^{(3)} \le \left\| W \right\|_{L^2(I^2)} \|u_n(t,\cdot)- v_n(t,\cdot) \|.
\ee
Finally, using (\textbf{W-2}), we estimate
\be\lbl{R-4}
R^{(4)}\le W_2 \|u_n(t,\cdot)- v_n(t,\cdot) \|.
\ee

The combination of \eqref{use-triangle}-\eqref{R-4} yields
\be\lbl{uniformly-Lip}
  \|R_n(u_n(t,\cdot))-R_n(v_n(t,\cdot))\|\le \left(L_f +L_D \left(\|W\|_{L^2(I^2)} +W_2^\prime\right) \right)
\|u_n(t,\cdot)- v_n(t,\cdot) \|,
\ee
i.e., $R_n$ is uniformly Lipschitz continuous. 
Recall that \eqref{dun} with the step functions \eqref{step-u} and \eqref{step-W}
is equivalent to the system of ordinary differential equations \eqref{gKM}. 
In turn, \eqref{uniformly-Lip} is equivalent to Lipschitz continuity of 
the right-hand side of \eqref{gKM}  with respect to discrete $L^2$-norm.
Thus, for every $n\in N,$  the IVP  \eqref{gKM}, \eqref{gKM-ic} has a unique solution,
which can be extended to $\R$.
\end{proof}

\subsection{A priori estimates} \lbl{sec.apriori}

Denote
$$
F:=\max_{t\in [0,T]} |f(0,t)|
$$
and recall \eqref{boundD}.

\begin{lem}\lbl{lem.apriori}
There exist positive constants $C_1$ and $C_2$ depending on $T$ but not on $n$,
such that
\be\lbl{apriori}
\max_{t\in [0,T]} \|\un(t)\| \le C_1\quad\mbox{and}\quad \max_{t\in [0,T]} \|\un^\prime(t)\| \le C_2,
\ee
uniformly in $n$.
\end{lem}
\begin{proof}{(Lemma~\ref{lem.apriori})}
Multiplying both sides of \eqref{dun} by $u_n(t,x)$ and integrating over $I$, we obtain
\be\lbl{bounding-un}
\begin{split}
{1\over 2}  {d\over dt} \| u_n(t,\cdot)\|^2 &\le \int_I |f(u_n(x,t),t)| |u_n(x,t)| dx 
+ \int_{I^2} | W_n(x,y)| | D\left(u_n(t,y)-u_n(t,x)\right)|  |u_n(t,x)| dxdy \\
& \le \int_I \left| f(u_n(x,t),t)-f(0,t)\right||u_n(x,t)| dx  + F\int_I |u_n(x,t)| dx \\
& + \int_{I^2} | W_n(x,y)| |u_n(t,x)| dxdy \\
&\le L_f \|u_n(t,\cdot)\|^2+ \left(F+\|W\|_{L^2(I^2)}\right) \left(\|u_n(t,\cdot)\|^2+1\right)\\
& \le \left(L_f +F+\|W\|_{L^2(I^2)}\right)\|u_n(t,\cdot)\|^2 + \left(F+\|W\|_{L^2(I^2)}\right),
\end{split}
\ee
where we used the Cauchy-Schwarz inequality and the bound 
$\|u_n(t,\cdot)\|\le \|u_n(t,\cdot)\|^2+1.$

Thus, 
\be\lbl{pre-G}
{d\over dt} \| u_n(t,\cdot)\|^2 \le C_3 \| u_n(t,\cdot)\|^2 +C_4,
\ee
with $C_3=2 \left(L_f +F+\|W\|_{L^2(I^2)}\right)$ and $C_4=2 \left(F+\|W\|_{L^2(I^2)}\right)$.
Using Gronwall's inequality and taking maximum over $t\in [0,T]$, we have
\be\lbl{bound-for-un}
\max_{t\in [0,T]} \| \un(t)\|^2\le e^{C_3T} \left( \|g\|^2 +C_4\right).
\ee
Here, we also used $\|\un(0)\|\le \|g\|$, because $\un(0)$ is an $L^2$-projection 
of $g$ onto $X_n$.

We now turn to bounding $\|\un^\prime(t)\|$. To this end, multiply \eqref{dun} by $v\in X$
and integrate both sides over $I$ to obtain
\begin{equation*}
\left( \un^\prime(t), v\right) = \int_I f\left(u_n(t,x)\right)v(x) dx+
\int\int_{I^2} W_n(x,y) D\left( u_n(t,x)-u_n(t,y)\right) v(x) dxdy.
\end{equation*}
Proceeding as in \eqref{bounding-un}, we obtain 
$$
\left|\left( \un^\prime(t), v\right)\right| \le \left(L_f+F +\|W\|_{L^2(I^2)} \right) \|v\|
\quad \forall v\in X.
$$
Thus, 
$$
\sup_{t\in \R} \|\un^\prime (t)\| \le C_2, \quad C_2:=L_f+F +\|W\|_{L^2(I^2)}.
$$  
\end{proof}

\subsection{Existence} \lbl{sec.exists}
With Lemma~\ref{lem.apriori} in hand, we are now ready to show existence of a weak
solution of \eqref{clim}. Furthermore, we show that the weak solution of 
\eqref{clim}
is the limit of the solutions of the discrete problems \eqref{dun}, i.e., the limit of solutions 
of \eqref{gKM}, \eqref{gKM-ic}. 

From Lemma~\ref{lem.apriori}, we have
\be\lbl{we-have}
\| \un\|_{C(0,T;X)} \le C_1,\quad \| \un(t+h)-\un(t)\| \le C_2 |h|.
\ee
From \eqref{we-have}, we further obtain
\be\lbl{pre-FK}
\| \un\|_{L^2(0,T;X)} \le C_1^2 T,\quad \int_0^T \| \un(t+h)-\un(t)\|^2dt \le C_2^2 h^2 T.
\ee
By the Frechet-Kolmogorov theorem \cite{Yosida-FA}, $\{\un\}$ is precompact in 
$L^2(0,T;X)$. Let $\{\mathbf{u}_{n_k}\}$ be a convergent subsequence of $\{\un\}$.
Denote its limit by $\bu$.

By Lemma~\ref{lem.apriori},
$$
\|\un^\prime\|_{L^2(0,T;X)} \le C_2\sqrt{T}.
$$
Therefore, $\{\mathbf{u}^\prime_{n_k}\}$ is weakly precompact in $L^2(0,T;X)$. 
Let $\{\mathbf{u}^\prime_{n_{k^\prime}}\}$ be a subsequence converging to 
$\mathbf{w}\in L^2(0,T;X)$. 

We show that $\mathbf{w}=\bu^\prime$.  Indeed, for arbitrary $\phi \in C_c^1(0,T)$
and $w\in X$, we have 
 \be\lbl{byparts} 
\int_0^T \left( \mathbf{u}_{n_{k^\prime}}^\prime(t), \phi(t)w\right) dt = -\int_0^T
\left( \mathbf{u}_{n_{k^\prime}}(t), \phi^\prime (t)w\right)  dt. 
\ee 
Sending $k^\prime \to\infty$ in (\ref{byparts}), and using
$
\mathbf{u}_{n_{k^\prime}}^\prime\rightharpoonup \mathbf{w}
$
and 
$
\mathbf{u}_{n_{k^\prime}}\rightharpoonup \bu,
$
we obtain
$$
\int_0^T\left(\mathbf{w}(t), \phi(t)w\right)=-\int_0^T\left(\mathbf{u}(t), \phi^\prime(t) w\right)dt.
$$
By \cite[Corollary~2]{Yosida-FA},
$$
\left( \int_0^T\mathbf{w}(t)\phi(t) dt, w\right)=\left(-\int_0^T\mathbf{u}(t)\phi ^\prime(t) dt, w\right)
\quad \forall w\in X.
$$
We conclude that $\mathbf{u}\in L^2(0,T;X)$ is weakly differentiable and
$\mathbf{u}^\prime=\mathbf{w}\in L^2(0,T;X).$ Thus, $\mathbf{u} \in H^1(0,T; X).$ 

Next, we show that $\mathbf{u} \in H^1(0,T; X)$ is a weak solution of \eqref{clim}, 
\eqref{clim-ic}. To this end, fix $N\in\N$ and choose a function of the form
\be\lbl{separable}
\mathbf{v}(t)=\sum_{j=1}^{N} d_j(t) \mathbf{\phi}_{N,j},
\ee
where $d_j(t)$ are continuously differentiable functions. Adding up \eqref{project} with
$n>N$ and $\mathbf{\phi}:=d_j(t)\phi_{nj}$ by $d_j(t),$ $j\in [n]$ and integrating the result
from $0$ to $T$, we obtain
\be\lbl{weak-Galerkin}
\int_0^T
(\mathbf{u^\prime}_n(t)-K(\mathbf{u}_n(t))-f(\mathbf{u}_n(t),t),
\mathbf{v}(t))dt=0,
\ee
where $\mathbf{v}$ is as in \eqref{separable}.
Passing to the limit along $n=n_k$, we have 
\be\lbl{w-limit}
\int_0^T (\mathbf{u}^\prime(t)-K(\mathbf{u}(t))-f(\mathbf{u}(t),t),\mathbf{v}(t)) dt=0. 
\ee 
This equality holds for an arbitrary $\mathbf{v}$ in the form of \eqref{separable}.
Since such functions for $N\in\N$ are dense in $L^2(0,T;X)$, we conclude
that \eqref{w-limit} holds for all $\mathbf{v}\in L^2(0,T;X)$.
Therefore, 
\be\lbl{weak-equality} 
(\mathbf{u^\prime}-K(\mathbf{u})-f(\mathbf{u},t),\mathbf{v})=0 \quad \forall 
\mathbf{v}\in L^2(0,T;X) \;\mbox{a.e. on}\; [0,T]
\ee 
In particular, \eqref{weak-equality} holds for any $\bv\in X$.

Next, we verify $\mathbf{u}(0)=g$.  From \eqref{weak-equality} for  
any $\mathbf{v}\in C^1(0,T;X)$ vanishing at $t=T$ 
via integration by parts we have
\be\lbl{ic1}
-\int_0^T\left(\mathbf{u}(t), \mathbf{v^\prime}(t)\right)dt=
\int_0^T \left( K(\mathbf{u}(t)) + f(\mathbf{u}(t),t),
\mathbf{v}(t)\right)dt+\left( \mathbf{u}(0), \mathbf{v}(0)\right).
\ee 
Likewise, by \eqref{weak-Galerkin},
\be\lbl{ic2} 
-\int_0^T\left(\mathbf{u}_{n_k}(t),\mathbf{v^\prime}(t)\right)dt= 
\left(K(\mathbf{u}_{n_k}(t)) + f(\mathbf{u}_{n_k}(t),t), \mathbf{v}(t)\right)dt+ 
\left(\mathbf{u}_{n_k}(0), \mathbf{v}(0)\right). 
\ee 
Passing to the limit (along a subsequence) in (\ref{ic2}) yields 
\be\lbl{ic2-limit}
-\int_0^T\left(\mathbf{u}(t), \mathbf{v^\prime}(t)\right)dt=
\int_0^T \left(K(\mathbf{u}(t)) + f(\mathbf{u}(t),t),
\mathbf{v}(t)\right)dt + \left( \mathbf{g}, \mathbf{v}(0)\right).
\ee 
As $\mathbf{v}(0)\in X$ is arbitrary, from \eqref{ic2} and \eqref{ic2-limit}
we conclude
$\mathbf{u}(0)=\mathbf{g}$. Thus, $\mathbf{u}$ is a weak solution of \eqref{clim}, 
\eqref{clim-ic}. 

\subsection{Uniqueness} \lbl{sec.unique}
Suppose the solution of the IVP \eqref{clim}, \eqref{clim-ic} is not unique. Then
there are two functions $\bu, \mathbf{w} \in H^1 (0,T; X)$ satisfying the same initial
condition $\bu(0)=\bv(0)$ and such that
\begin{eqnarray}
\lbl{solution-1}
(\mathbf{u^\prime}(t)-K(\mathbf{u}(t))-f(\mathbf{u}(t),t),\mathbf{v})&=&0,\\ 
\lbl{solution-2}
(\mathbf{w^\prime(t)}-K(\mathbf{w}(t))-f(\mathbf{w}(t),t),\mathbf{v})&=&0, \;\mbox{a.e. on}\; [0,T] . 
\end{eqnarray}
for any $\mathbf{v}\in L^2(0,T;X).$ 
Set $\mathbf{\xi}=\bu-\mathbf{w}$ and $\bv=\mathbf{\xi}$. After subtracting 
\eqref{solution-2} from \eqref{solution-1}, and using Lipschitz continuity of 
$f$ and $D$, we obtain
$$
{1\over 2} {d\over dt} \| \xi(t,\cdot)\|^2 \le L_f \|\xi(t,\cdot)\|^2 +
L_D\int_{I^2} \left| W(x,y)\right| \left( |\xi(t,y)|+ |\xi(t,x)|\right) |\xi(t,x)| dxdy.
$$
and, thus,
\be\lbl{Gronwall-for-uniqueness}
{d\over dt} \|\mathbf{\xi}(t)\|^2 \le \left( 2L_f+4L_D\|W\|_{L^2(I^2)}\right) \|\mathbf{\xi}(t)\|^2.
\ee
By Gronwall's inequality,
$$
\max_{t\in [0,T]} \|\mathbf{\xi}(t)\|^2 \le e^{ \left( 2L_f+4L_D\|W\|_{L^2(I^2)}\right)T} 
\|\mathbf{\xi}(0)\|^2=0.
$$
Thus, $\bu=\mathbf{w}$. By contradiction, there is a unique weak solution of the 
IVP \eqref{clim}, \eqref{clim-ic}.

The uniqueness of the weak solution entails $\mathbf{u}_n\to\mathbf{u}$
as $n\to\infty$. Indeed, suppose on the contrary that there exists a subsequence
$\mathbf{u}_{n_l}$ not converging to $\mathbf{u}$. Then for a given $\epsilon>0$
one can select a subsequence $\mathbf{u}_{n_{l_i}}$ such that
$$
\|\mathbf{u}_{n_{l_i}} -\mathbf{u}\|_{L^2(0,T;X)}>\epsilon \;\forall i\in\N.
$$
However, $\{\mathbf{u}_{n_{l_i}}\}$ is precompact in $L^2(0,T,X)$ and contains a
subsequence converging to a weak solution of \eqref{clim}, which must be
$\mathbf{u}$ by uniqueness. Contradiction.

\subsection{Convergence of solutions of the averaged equation}

We now show that like the solutions of the Galerkin problems, the solutions of the 
IVP for the averaged equation \eqref{dvn}, \eqref{dvn-ic} converge to the solution of the 
IVP for the continuum limit \eqref{clim}, \eqref{clim-ic}.

First, we need to develop several auxiliary estimates. For the truncated function $\tilde W$,
we have 
\be\lbl{bound-trunc}
\|\tilde W_n\|_{L^2(I^2)} \le \| W\|_{L^2(I^2)}.
\ee

\begin{lem}\lbl{lem.L2converge}
$$
\lim_{n\to\infty}\| \bar W_n - W\|_{L^2(I^2)}=0.
$$
\end{lem}
\begin{proof}
Since $W_n\to W$ in $L^2$-norm, it is sufficient to show that $\|\bar W_n- W_n\|_{L^2(I^2)}$
tends to $0$ as $n\to\infty$. 

Let $\epsilon>0$ be given. Since $W\in L^2(I^2)$, there is $\delta>0$ such that
\be\lbl{uni-cont}
\int_A W^2< \epsilon^2
\ee
for any $A\subset I^2$ of Lebesgue measure $|A|<\delta$.
For a given $\lambda>0,$ denote $A_\lambda=\{ (x,y)\in I^2:\; W(x,y)>\lambda\}$.
Since $W\in L^1(I^2)$, $W$ is finite a.e., i.e.,
there exists $\lambda>0$ such that
\be\lbl{choose-lambda}
|A_\lambda|\le \delta.
\ee
Let $N_\lambda\in\N$ such that
\be\lbl{choose-N}
\alpha^{-1}_n\ge \lambda\quad n\ge N_\lambda.
\ee

For $n\ge N_\lambda,$ we have
\begin{equation*}
\begin{split}
\| \bar W_n -W_n\|^2_{L^2(I^2)} 
&= \sum_{i,j=1}^n \int_{I_{n,i}\times I_{n,j}} \left( \bar W_n-W_n\right)^2 \\
&= \sum_{i,j=1}^n n^{-2} \left( n^2 \int_{I_{n,i}\times I_{n,j}} \left( \tilde W_n-W \right) \right)^2 \\
& = \sum_{i,j=1}^n n^2 \left(\int_{I_{n,i}\times I_{n,j}} \left( \tilde W_n-W \right) \right)^2 \\
&\le \sum_{i,j=1}^n \int_{I_{n,i}\times I_{n,j}} \left( \tilde W_n-W \right)^2\\
&= \int_{I^2} \left( \tilde W_n-W \right)^2= \int_{A_\lambda}\left( \tilde W_n-W \right)^2\\
&\le \int_{A_\lambda} W^2 \le \epsilon^2.
\end{split}
\end{equation*}
\end{proof}

Further, let 
\be\lbl{def-Kn}
K_n\left(  v \right) = \int_I \bar W_n(\cdot, y) D\left( v(y)-v(\cdot)\right) dy
\ee
be a nonlinear map from $X$ to itself.

\begin{lem}\lbl{lem.Kcont}
$K_n$ is a uniformly Lipschitz continuous map from $X$  to itself
\be\lbl{KLip}
\| K_n(v)-K_n(u)\| \le L_{K} \| v-u\| \quad \forall u, v\in X,
\ee
where $L_K =2 \|W\|_{L^2(I^2)} L_D.$ In addition,
\be\lbl{Kconverge}
\| K_n(v)-K(v)\|\le \|\bar W_n -W\|_{L^2(I^2)}\quad \forall v\in X.
\ee
\end{lem} 
\begin{proof}
Using Lipschitz continuity of $D$, Cauchy-Schwartz inequality, and 
\eqref{bound-trunc}, we have
\begin{equation*}
\begin{split}
\| K_n(u)-K_n(v)\| & \le \left\| \int_I \bar W_n(\cdot,y) \left\{ D\left( u(y)- u(\cdot)\right)
-D\left( v(y)- v(\cdot)\right) \right\} dy \right\|\\
& \le L_D\left\{ \left\| \int_I \bar W_n(\cdot,y)  \left| u(y)- v(y) \right| dy \right\| 
+ \left\| \int_I \bar W_n(\cdot,y)  \left| u(\cdot)- v(\cdot) \right| dy \right\| \right\}\\
&\le 2L_D \| W\|_{L^2(I^2)} \|u-v\|.
\end{split}
\end{equation*}
To show \eqref{Kconverge}, we use \eqref{boundD} and Cauchy-Schwartz inequality:
\begin{equation*}
\begin{split}
\| K_n(v)-K(v)\| &\le \left\| \int_I\left( \bar W_n(\cdot,y) -W(\cdot,y)\right) D\left(v(y)-v(\cdot)\right)
dy\right\|\\
&\le \left\| \int_I\left| \bar W_n(\cdot,y) -W(\cdot,y)\right| dy\right\|\\
&\le \|\bar W_n -W\|_{L^2(I^2)}.
\end{split}
\end{equation*}
\end{proof}

We rewrite the averaged equation \eqref{dvn} as 
\be\lbl{weak-ave}
\left( \bv_n^\prime (t) -K_n( \bv_n(t)) -f(\bv(t)),\phi \right) =0\quad \forall \phi\in X_n.
\ee
subject to the initial condition
\be\lbl{weak-ave-ic}
\bv_n (0) =\sum_{i=0}^n g_{n,i} \phi_{n,i}.
\ee
We want to show that $\bv_n\to \bu$ in $L^2(0,T; X)$. 
To this end, note that a priori estimates in \S\ref{sec.apriori} hold for the averaged problem
\eqref{weak-ave-ic} due to \eqref{bound-trunc}. The rest of the proof 
is done by following the lines 
of the existence and uniqueness proof in \S\S~\ref{sec.exists}, \ref{sec.unique}.
The only place, which requires a clarification  is the following limit \footnote{
This limit is used in \eqref{w-limit} and \eqref{ic2-limit}.}.
\begin{lem}\lbl{lem.clarify}
\be\lbl{clarify-the-limit}
\int_0^T \left( K_n(\bv_n(t)), \bv(t) \right)dt \to \int_0^T \left( K (\bv (t)), \bv (t) \right) dt 
\ee
for any $v\in C^1(0,T; X),$ provided that $\bv_n\to \bu$ in $L^2(0,T; X)$.
\end{lem}
\begin{proof}
\be\lbl{I1+I2}
\begin{split}
\left|\int_0^T \left( K_n(\bv_n(t)) -K(\bu(t)), \bv(t) \right) dt \right| &\le 
\int_0^T \left| \left( K_n(\bv_n(t))-K_n(\bu(t)), \bv(t)\right) \right| dt\\
& +\int_0^T \left| \left( K_n(\bu(t))-K(\bu(t)), \bv(t)\right) \right| dt 
=: I_1+I_2.
\end{split}
\end{equation}
Using \eqref{KLip} and Cauchy-Schwartz inequality, we have
\be\lbl{do-I1}
\begin{split}
I_1 &= \int_0^T \| K_n(\bv_n(t))-K_n(\bu(t))\| \|\bv(t)\| dt \\
&\le L_K \left( \int_0^T \|\bv_n(t)-\bu(t)\|^2dt \right)^{1/2} \|\bv \|_{L^2(0,T;X)}\\
& \le L_K \|\bv_n(t)-\bu(t)\|_{L^2(0,T;X)} \|\bv \|_{L^2(0,T;X)}.
\end{split}
\ee
Similarly, using \eqref{Kconverge} and the Cauchy-Schwartz inequality, we further obtain
\be\lbl{do-I2}
\begin{split}
I_2 &= \int_0^T \| K_n(\bu(t))-K(\bu(t))\| \|\bv(t)\| dt \\
& \le \|\bar W_n-W\|_{L^2(I^2)} \|\bv \|_{L^2(0,T;X)}.
\end{split}
\ee
Plugging \eqref{do-I1} and \eqref{do-I2} in  \eqref{I1+I2}, we obtain
\begin{equation*}
\begin{split}
\left|\int_0^T \left( K_n(\bv_n(t)) -K(\bu(t)), \bv(t)\right) dt \right| & \le
\left(L_K \| \bv_n(t)-\bu(t)\|_{L^2(0,T;X)} \right.\\
&\left. +\|\bar W_n-W\|_{L^2(I^2)} \right)
\|\bv \|_{L^2(0,T;X)}.
\end{split}
\end{equation*}
The statement of the lemma follows the above inequality and Lemma~\ref{lem.L2converge}.
\end{proof}

\vskip 0.2cm
\noindent
{\bf Acknowledgements.}
The author thanks Gideon Simpson for insightful comments, which helped to improve the manuscript. 
This work was supported in part by the NSF grant DMS 1715161.

\vfill\newpage

\def\cprime{$'$} \def\cprime{$'$} \def\cprime{$'$}
\providecommand{\bysame}{\leavevmode\hbox to3em{\hrulefill}\thinspace}
\providecommand{\MR}{\relax\ifhmode\unskip\space\fi MR }
\providecommand{\MRhref}[2]{%
  \href{http://www.ams.org/mathscinet-getitem?mr=#1}{#2}
}
\providecommand{\href}[2]{#2}

\end{document}